\documentclass[12pt, reqno]{amsart}
\usepackage{amsmath, amsthm, amscd, amsfonts, amssymb, graphicx, color}
\usepackage[bookmarksnumbered, colorlinks, plainpages]{hyperref}
\hypersetup{colorlinks=true,linkcolor=red, anchorcolor=green, citecolor=cyan, urlcolor=red, filecolor=magenta, pdftoolbar=true}

\newtheorem{theorem}{Theorem}[section]

\newtheorem{proposition}[theorem]{Proposition}

\theoremstyle{definition}
\newtheorem{definition}[theorem]{Definition}

\theoremstyle{remark}
\newtheorem{remark}[theorem]{Remark}
\newtheorem{note}[theorem]{Note}
\numberwithin{equation}{section}

\begin{document}

\setcounter{page}{1}

\title[Fusion frame duals]{Dual fusion frames in the sense of Kutyniok, Paternostro and Philipp with applications to invertible Bessel fusion multipliers}

\author[H. Javanshiri, A. Fattahi, \MakeLowercase{and} M. Sargazi]{Hossein Javanshiri$^1$, Abdolmajid Fattahi$^2$ \MakeLowercase{and} Mojtaba Sargazi$^3$}

\address{$^{1}$ Department of Mathematics, Yazd University, Iran.}
\email{\textcolor[rgb]{0.00,0.00,0.84}{h.javanshiri@yazd.ac.ir;}}

\address{$^{2,3}$ Department of Mathematics, Kermanshah University, Iran.}
\email{\textcolor[rgb]{0.00,0.00,0.84}{majidzr@razi.ac.ir;}}

%\address{$^{3}$ Department of Mathematics, Kermanshah University, Iran.}
\email{\textcolor[rgb]{0.00,0.00,0.84}{sargazi1984@gmail.com.}}

%\dedicatory{This paper is dedicated to Professor ABCD}

%\let\thefootnote\relax\footnote{Copyright 2017 by the Tusi Mathematical Research Group.}

\subjclass[2010]{Primary 42C15, 41A58; Secondary 46C05, 42C40.}

\keywords{Fusion frame, operator-valued frame, fusion frame dual, Bessel fusion multiplier, Riesz basis.}

%\date{Received: xx/xx; Revised: xx/xx; Accepted: xx/xx.
%\newline \indent $^{*}$Corresponding author

\begin{abstract}
To achieve our main research goal, first we survey the approaches towards dual fusion frames existing in the literature and agree on the notion of duality for fusion frames in the sense of Kutyniok, Paternostro and Philipp ({\it Oper. Matrices} {\bf11} (2017), no. 2, 301--336). As a main result we show that different fusion frames have different dual fusion frames. Moreover, this duality notion leads to a new definition of Bessel fusion multipliers
which is a slightly modified version of the commonly used definitions. Particularly, we show that with this definition in many cases Bessel fusion multipliers behave similar to ordinary Bessel multipliers.
Finally, special attention is devoted to the study of dual fusion frames induced by an invertible Bessel fusion multiplier.
\end{abstract} \maketitle

%%%%%%%%%%%%%%%%%%%%%%%%%%%%%%%%%%%%%%%%%%%%%%%%%%%%%%%%%%%%%%%%%%%%%%%%%%%%%%%%%%%%%%%
%%%%%%%%%%%%%%%%%%%%%%%%%%%%%%%%%%%%%%%%%%%%%%%%%%%%%%%%%%%%%%%%%%%%%%%%%%%%%%%%%%%%%%%

\section{Introduction}

Due to the fundamental work done by Casazza and Kutyniok \cite{6},
fusion frames (or frame of subspaces) was formally introduced and popularized from
then on. Fusion frames play an important role in theory and applications; In details,
recent studies shows that fusion frames provide effective frameworks for modeling of sensor networks, signal and image processing, filter banks and variety of applications that cannot be modeled by ordinary frames \cite{ap1,ap2,ap3}.
The essence of fusion frames is that we can first build ordinary frames in subspaces, called local frames, and then piece them together to obtain ordinary frames for the whole space. For some applications it is important to find dual fusion frames of a given fusion frame to reconstruct the modified data and compare it with the original data. In this paper, we are going to give new results on the duality of fusion frames as well as dual fusion frames induced by an invertible Bessel fusion multiplier.
This results are not
only of interest on their own, but they also paves the way
to have different reconstruction strategies in terms of fusion frames.
First, let us to review the notions of duality for fusion frames existing in the literature. To this end, we need to recall some notations and definitions.

Throughout this paper, we denote by $\mathcal H$ a separable Hilbert space with the inner
product $\big<\cdot,\cdot\big>$ and $I$ refers to a finite
or countable index set. The notation $\mathcal W$ is used to denote the sequence $\{W_i\}_{i\in I}$ of closed subspaces of $\mathcal{H}$ and $\omega$ denotes a sequence of weights $\{\omega_{i}\}_{i\in I}$, that is, $\omega_{i}\geq 0$. Also, $(\mathcal{W},\omega)$ refers to the fusion sequence $\{(W_i,\omega_{i})\}_{i\in I}$, that is,
$$W_i=\{0\}\quad\Longleftrightarrow\quad \omega_i=0\quad\quad\quad(i\in I).$$
Recall from \cite{6} (and also \cite{16}) that the fusion sequence $(\mathcal{W},\omega)$ is a fusion frame for ${\mathcal H}$, if there exist constants $0<\alpha_{{\mathcal W}_{\omega}}\leq \beta_{{\mathcal W}_{\omega}}<\infty$ such that
\begin{equation}\label{p1}
\alpha_{{\mathcal W}_{\omega}}\|x\|^{2} \leq \sum_{i\in I} \omega_{i}^{2} \|P_{W_i}x\|^{2} \leq \beta_{{\mathcal W}_{\omega}}\|x\|^{2}\quad\quad\quad\quad\quad (x\in {\mathcal H}),
\end{equation}
where $P_{W_i}$ denotes the orthogonal projection from $\mathcal H$ onto ${W_i}$.
The constants $\alpha_{{\mathcal W}_{\omega}}$ and $\beta_{{\mathcal W}_{\omega}}$ are called lower and upper fusion frame bound, respectively. If only the right inequality of (\ref{p1}) is
considered, then $(\mathcal{W},\omega)$ is called a Bessel fusion sequence. Recall also from \cite{6} that for each sequence $\{W_i\}_{i\in I}$ of closed subspaces of $\mathcal{H}$, the space
\begin{equation}\label{h-parsa}
{\mathcal K}_{\mathcal W}:={\Big(}\sum_{i\in I}\oplus W_i{\Big)}_{\ell^2}=\Big\{ \{z_{i}\}_{i\in I}:~z_{i} \in W_i~~{\hbox{and}}~\sum_{i\in I}\|z_i\|^2<\infty\Big\},
\end{equation}
with the inner product $\langle \{z_{i}\}_{i\in I} , \{y_{i}\}_{i\in I} \rangle = \sum_{i\in I} \langle z_{i} , y_{i} \rangle$ is a Hilbert space.
For a Bessel fusion sequence $(\mathcal{W},\omega)$ of $\mathcal H$, the analysis operator $T_{\mathcal W}$ is defined by
$$T_{\mathcal W}:{\mathcal H}\rightarrow{\mathcal K}_{\mathcal W};\quad\quad x\mapsto \{\omega_i P_{W_i}(x)\}_{i\in I}.$$
Its adjoint $T^*_{\mathcal W}$, the
synthesis operator of $(\mathcal{W},\omega)$, maps ${\mathcal K}_{\mathcal W}$ into $\mathcal H$ and defined by $T^*_{\mathcal W}(\{x_i\}_{i\in I})=\sum_{i\in I}\omega_i x_i$. Moreover, the fusion frame operator $S_{\mathcal W}:{\mathcal H}\rightarrow {\mathcal H}$ is defined by $S_{\mathcal W}(x)=\sum_{i\in I}\omega_i^2 P_{W_i}(x)$, which is a bounded, invertible and positive operator; see \cite{6} for more information.

Now, we are in position to discuss and compare the notions of duality for fusion frames existing in the literature. As far as we know the subject, the first definition was presented by G$\check{a}$vru$\c{t}$a \cite{13}. This paper initiated a series of subsequent publications and has had a great impact; see for example \cite{gav2,gav3}. A Bessel fusion sequence $({\mathcal V},\upsilon)=\{V_i,\upsilon_i\})_{i\in I}$ is called a G$\check{a}$vru$\c{t}$a-dual fusion frame of $({\mathcal W},\omega)$ if
\begin{eqnarray}\label{p2}
f=\sum_{i\in I}\omega_i\upsilon_i P_{V_i}S^{-1}_{\mathcal W}P_{W_i}(f)\quad\quad\quad(f\in{\mathcal H}).
\end{eqnarray}
It was shown in \cite{6} that a G$\check{a}$vru$\c{t}$a-dual $\mathcal V$ of $\mathcal W$ is itself a fusion frame. We should however note that it is in general not true that $\mathcal W$ is a G$\check{a}$vru$\c{t}$a-dual of $\mathcal V$. Moreover, we would
like to mention that, in contrast to ordinary frames, equality (\ref{p2}) cannot be generally expressed in terms of corresponding synthesis and analysis operators. Motivated by these facts a general approach to dual fusion frames has been proposed by Heineken et al. \cite{14}. More precisely, in \cite{14}, a Bessel fusion sequence $({\mathcal V},\upsilon)$ was called a dual fusion frame of $({\mathcal W},\omega)$ if there exists a bounded operator
$$Q:{\mathcal K}_{\mathcal W}\rightarrow
{\mathcal K}_{\mathcal V}$$
such that $T_{\mathcal V}QT^*_{\mathcal W}=Id_{\mathcal H}$ or equivalently
\begin{equation}\label{q-parsa}
x=\sum_{i\in I}\upsilon_i{\Big(}Q(\omega_jP_{W_j}x){\Big)}_i\quad\quad\quad(x\in{\mathcal H}),
\end{equation}
where here and in the sequel
$Id_{\mathcal H}$ is the identity operator on $\mathcal H$. Here, we shall call these fusion sequences HMBZ-duals.
This notion of duality paves the way to obtain results
which are analogous to those valid for ordinary dual frames. Moreover, the authors of \cite{14} present a very good analysis of the duality relation between two Bessel fusion sequences. They in particular characterize the set of all component preserving HMBZ-duals of a certain fusion frame; see Theorem 3.11 of \cite{14}. As we understand from \cite[Subsection 3.2]{16}, any two fusion frames for $\mathcal H$ are HMBZ-duals.
In details, there is too much freedom in the choice of the operator $Q$ in
the definition of HMBZ-duals. Motivated by this fact,
Kutyniok et al. \cite{16} considered a modified version of HMBZ-duals which is a generalization of the idea in \cite{hein}.

This notion of duality for fusion frame in the sense of \cite{16}, as the main object of study
of this work, will be explained and justified in the next sections.
Moreover, based on this duality notion for fusion frames, we
propose a new concept of Bessel fusion multipliers in Hilbert spaces, which
is a slightly modified version of \cite{ar,msh}. It extends the commonly used notion
and, in particular, we show that with this definition in many cases Bessel fusion
multipliers behave similar to ordinary Bessel multipliers.

%%%%%%%%%%%%%%%%%%%%%%%%%%%%%%%%%%%%%%%%%%%%%%%%%%%%%%%%%%%%%%%%%%%%%%%%%%%%%%%%%%%%%%
%%%%%%%%%%%%%%%%%%%%%%%%%%%%%%%%%%%%%%%%%%%%%%%%%%%%%%%%%%%%%%%%%%%%%%%%%%%%%%%%%%%%%%

\section{Operator-valued frames: an overview}

In this section, we give a brief overview of operator-valued frames which includes ordinary frames and fusion frames as elementary examples.
To this end, we need to set some notations that will be used throughout the paper.
In what follows, $\ell^p(I)$ ($1\leq p<\infty$), $\ell^\infty(I)$ and $c_0(I)$ have their usual meanings and $\mathcal H$ and
$\mathcal K$ always denote Hilbert spaces.
The set of all bounded linear operators from $\mathcal H$ into $\mathcal K$ will be denoted by $B({\mathcal H},{\mathcal K})$. As usual, we set $B({\mathcal H}):=B({\mathcal H},{\mathcal H})$. An operator $T\in B({\mathcal H},{\mathcal K})$ is said to be in the Schatten $p$-class if $\{\lambda_i\}_{i\in I}\in\ell^p(I)$ where $\{\lambda_i\}_{i\in I}$ is the sequence of positive
eigenvalues of $|T|=(T^*T)^{1/2}$ arranged in decreasing order and repeated according to multiplicity. Given $1\leq p<\infty$,
$C_p({\mathcal H},{\mathcal K})$ denotes the Schatten $p$-class.
For an operator $T\in B({\mathcal H},{\mathcal K})$,
\begin{itemize}
\item the notations ${\rm ran}(T)$ and $\ker(T)$ are used to denote the range and the kernel of $T$, respectively
\item the letter $T|_V$ refers to the restriction of $T$ to a subspace $V\subset {\mathcal H}$.
\end{itemize}
Following (\ref{h-parsa}), for Hilbert spaces $\mathcal H$ and $\mathcal K$, we set
\begin{eqnarray*}
\mathfrak{H}:={\Big(}\sum_{i\in I}\oplus {\mathcal H}{\Big)}_{\ell^2}\quad\quad\quad{\hbox{and}}\quad\quad\quad    \mathfrak{K}:={\Big(}\sum_{i\in I}\oplus {\mathcal K}{\Big)}_{\ell^2}.
\end{eqnarray*}
Moreover, the notation $\ell^\infty(I,B({\mathcal H}))$ is used to denote the set $${\Big\{}{\mathcal R}=\{R_i\}_{i\in I}\subset B({\mathcal H}):~\|{\mathcal R}\|_\infty:=\sup_{i\in I}\|R_i\|<\infty{\Big\}},$$
and $C_0(I,B({\mathcal H}))$ refers to the set of all
${\mathcal R}\subset B({\mathcal H})$ with the property that for each $\varepsilon>0$ there exists a finite set $J\subseteq I$ with $\sup_{i\in I\setminus J}\|R_i\|<\varepsilon$.

Recall from \cite{16} (and also \cite{kaf}) that a sequence
${\mathcal A}=\{A_i\}_{i\in I}\subset B({\mathcal H},{\mathcal K})$ is an operator-valued (or $B({\mathcal H},{\mathcal K})$-valued) frame, if there exist constants $\alpha_{\mathcal A}$ and $\beta_{\mathcal A}$ such that
\begin{equation}\label{p01}
\alpha_{\mathcal A}\|x\|^2\leq\sum_{i\in I}\|A_ix\|^2\leq \beta_{\mathcal A}\|x\|^2\quad\quad\quad(x\in{\mathcal H}).
\end{equation}
The constants $\alpha_{\mathcal A}$ and $\beta_{\mathcal A}$ are called lower and upper frame bound of $\mathcal A$, respectively.
If only
the right inequality of (\ref{p01}) is considered, then
$\mathcal A$ is called an operator-valued Bessel sequence.

For an operator-valued Bessel sequence ${\mathcal A}\subset B({\mathcal H},{\mathcal K})$, the notation $T_{\mathcal A}:{\mathcal H}\rightarrow \mathfrak{K}$ with $T_{\mathcal A}x:=\{A_ix\}_{i\in I}$ denotes the
associated analysis operator. Its adjoint $T_{\mathcal A}^*$, the synthesis operator of $\mathcal A$, maps $\mathfrak{K}$ into $\mathcal H$ and defined by $T_{\mathcal A}^* \{z_i\}_{i\in I}:=\sum_{i\in I} A_i^*z_i$ for all $\{z_i\}_{i\in I}\in{\mathfrak{K}}$.
The operator $S_{\mathcal A}:=T_{\mathcal A}^* T_{\mathcal A}$ is called the frame operator corresponding to ${\mathcal A}$ and it is a positive bounded self-adjoint operator.
In particular, for an operator-valued frame ${\mathcal A}$ the operator $S_{\mathcal A}$ is invertible.

It is worthwhile to recall from \cite[Subsections 2.2 and 2.3]{16} that ordinary frames as well as fusion frames
can be considered as operator-valued frames. In details, suppose that $\Phi=\{\varphi_i\}_{i\in I}$ is a sequence (of vectors) in $\mathcal H$. If, for every $i\in I$, we define an operator $A_i^\Phi\in B({\mathcal H},{\mathbb C})$ by $A_i^\Phi x:=\big<x,\varphi_i\big>$ ($x\in {\mathcal H}$). Then, it is clear that ${\mathcal A}_{\Phi}:=\{A_i^\Phi\}_{i\in I}$ is an operator-valued frame if and only if $\Phi$ is a frame for $\mathcal H$, that is, there exist $\alpha_{{\mathcal A}_{\Phi}}, \beta_{{\mathcal A}_{\Phi}}>0$ such that
\begin{equation}
\alpha_{{\mathcal A}_{\Phi}}\|x\|^2\leq\sum_{i\in I}|\big<x,\varphi_n\big>|^2\leq \beta_{{\mathcal A}_{\Phi}}\|x\|^2\quad\quad\quad(x\in{\mathcal H}).
\end{equation}
In particular, the analysis and frame operator of the Bessel sequence $\Phi$ are defined by $T_\Phi:=T_{{\mathcal A}_\Phi}$ and $S_\Phi:=S_{{\mathcal A}_\Phi}$, respectively.
The reader will remark that if ${\mathcal K}={\mathbb C}$, then $\mathfrak{K}=\ell^2(I)$. Hence the above definitions is consistent with the corresponding definitions in the concept of ordinary frames; see \cite[Page 9 and 10]{16} for more information.
Similarly, a fusion sequence $({\mathcal W},\omega)$ can be identified with the sequence of operators ${\mathcal A}_{\mathcal W}:=\{\omega_iP_{W_i}\}_{i\in I}$ which is
a $B({\mathcal H})$-valued frame for $\mathcal H$ if and only if $({\mathcal W},\omega)$ is a fusion frame for $\mathcal H$. The analysis operator and the fusion frame operator of the Bessel fusion sequence $({\mathcal W},\omega)$ as an operator-valued Bessel sequence
are then defined by $T_{{\mathcal W},\omega}:= T_{{\mathcal A}_{\mathcal W}}$ and $S_{{\mathcal W},\omega}:=S_{{{\mathcal A}_{\mathcal W}}}$, respectively.
At this point we would like to remark that $T_{{\mathcal W},\omega}$ is an operator from $\mathcal H$ into $\mathfrak{H}$ whereas the operator $T_{\mathcal W}$ is an operator from $\mathcal H$ into ${\mathcal K}_{\mathcal W}$, by what was mentioned in Section 1.
With this notation, we can give a characterization of fusion frames in terms of the associated synthesis and
analysis operators. In details, the fusion sequence $({\mathcal W},\omega)$ is a
\begin{itemize}
\item Bessel fusion sequence if and only if $T_{{\mathcal W},\omega}^*$ is well-defined and bounded with $\|T_{{\mathcal W},\omega}^*\|\leq\sqrt{\beta_{{\mathcal W}_\omega}}$;
\item fusion frame if and only if $T_{{\mathcal W},\omega}^*$ is onto;
\item Riesz basis if and only if $T_{{\mathcal W},\omega}^*$ is invertible, that is, injective.
\end{itemize}

Recall also from \cite{16} that an operator-valued Bessel sequence ${\mathcal B}=\{B_i\}_{i\in I}\subset B({\mathcal H},{\mathcal K})$ is a dual operator-valued frame (or simply a dual) of an operator-valued frame
${\mathcal A}=\{A_i\}_{i\in I}$ if
$$T^*_{\mathcal B}T_{\mathcal A}x=\sum_{i\in I}B_i^*A_ix=x\quad\quad\quad(x\in{\mathcal H}).$$
By $D({\mathcal A})$ we denote the set of all duals of the operator-valued frame $\mathcal A$. In particular, it is shown in \cite[Lemma 3.2]{16} that
$$D({\mathcal A})={\Big\{}T_{\mathcal A}S_{\mathcal A}^{-1}+L:~L\in{\mathcal L}_{\mathcal A}{\Big\}},$$
where ${\mathcal L}_{\mathcal A}:=\{L\in B({\mathcal H},{\mathfrak{K}}):~L^*T_{\mathcal A}=0\}$.
From now on, we denote by $\widetilde{{\mathcal A}}(L)$ the dual of $\mathcal A$ corresponding to $L\in {\mathcal L}_{\mathcal A}$. Particularly, $\widetilde{{\mathcal A}}(0)$ is the canonical dual of $\mathcal A$ and
$$T_{\widetilde{A}(L)}=T_{\mathcal A}S_{\mathcal A}^{-1}+L,$$
for all $L\in {\mathcal L}_{\mathcal A}$.

%%%%%%%%%%%%%%%%%%%%%%%%%%%%%%%%%%%%%%%%%%%%%%%%%%%%%%%%%%%%%%%%%%%%%%%%%%%%%%%%%%%%%%
%%%%%%%%%%%%%%%%%%%%%%%%%%%%%%%%%%%%%%%%%%%%%%%%%%%%%%%%%%%%%%%%%%%%%%%%%%%%%%%%%%%%%%

\section{Main results}

We start this section by stating the definition of dual fusion frames in the sense of \cite{16}. It is essentially \cite[Definition 3.10]{16}, although
that definition is stated in terms of identity operator.

In this definition and in the sequel, for two fusion sequences
$({\mathcal V},\upsilon)$ and $({\mathcal W},\omega)$ in $\mathcal H$, the notation $I_0({\mathcal V},{\mathcal W})$ is used to denote the following subset of $I$:
\begin{align*}
{\Big\{}i\in I:~V_i=\{0\}~{\hbox{or}}~ W_i=\{0\}{\Big\}},
\end{align*}
which is equal to the set of all $i\in I$ such that $\upsilon_i=0$ or $\omega_i=0$. Moreover,
a sequence $\{Q_{i}\}_{i\in I}\subset B({\mathcal H})$ will be called $({\mathcal V},{\mathcal W})$-admissible if for each $i\in I$
\begin{eqnarray*}
W_i^\perp\subset\ker(Q_i),\quad {\rm ran}(Q_i)\subset V_i,\quad{\hbox{and}}\quad\|Q_i\|=1\quad{\hbox{if}}\quad i\notin I_0({\mathcal V},{\mathcal W}).
\end{eqnarray*}

\begin{definition}
Let $(\mathcal{W},\omega)$ be a fusion frame for ${\mathcal H}$. A Bessel fusion sequence $(\mathcal{V},\upsilon)$ is a (resp. generalized) fusion frame dual of $(\mathcal{W},\omega)$ if there exists a $({\mathcal V},{\mathcal W})$-admissible sequence $\{Q_{i}\}_{i\in I}\subset B({\mathcal H})$ such that $T_{\mathcal{V},\upsilon}^{*}D_QT_{\mathcal{W},\omega}$
is (resp. identity) invertible operator on $\mathcal H$, where
$$D_Q:{\mathfrak{H}}\rightarrow{\mathfrak{H}};\quad \{z_i\}_{i\in I}\mapsto\{Q_i z_i\}_{i\in I}.$$
\end{definition}

%%%%%%%%%%%%%%%%%%%%%%%%%%%%%%%%%%%%%%%%%%%%%%%%%%%%%%%%%%%%%%%%%%%%%%%%%%%%%%%%%%%%%%%%%%%%%%%%%%%%%%%%%%

At first we note that an argument similar to the proof of \cite[Theorem 3.13.]{16} with the aid of
\cite[Lemma 3.2.]{16} and \cite[Theorem 2.1]{j}  gives the following analogue of that theorem for generalized fusion frame duals, so we avoid the burden of proof.

\begin{proposition}\label{p8}
Let $(\mathcal{W},\omega)$ be a fusion frame for $\mathcal H$ and $(\mathcal{V},\upsilon)$ a Bessel fusion sequence. Then the following assertions are equivalent:
\begin{enumerate}
\item $(\mathcal{V},\upsilon)$ is a fusion frame dual of $(\mathcal{W},\omega)$.
\item There exists a $({\mathcal V},{\mathcal W})$-admissible  sequence $\{Q_{i}\}_{i\in I}\subset B({\mathcal H})$ such that the Bessel sequence  $\{\upsilon_{i}Q_{i}^{*}\}_{i\in I}$ is a generalized dual operator-valued frame of ${\mathcal A}_{\mathcal W}$.
\item There exists an operator-valued Bessel sequence $\mathcal{L}=\{L_{i}\}_{i\in I}\subset B({\mathcal H})$ with $T_{\mathcal{L}}^{*}T_{\mathcal{W},\omega}=0$ and invertible operator $U\in B({\mathcal H})$ such that for the operators $A_{i}:=( \omega_{i}US_{\mathcal{W},\omega}^{-1}+L_{i}^{*}) P_{{W}_{i}}$ we have ${\rm ran} A_{i} \subset {V}_{i}$ for all $i\in I$ and if $i \notin I_{0}(\mathcal{V},\mathcal{W})$, $\|A_{i}\|=\upsilon_{i}$.
\end{enumerate}
\end{proposition}

%%%%%%%%%%%%%%%%%%%%%%%%%%%%%%%%%%%%%%%%%%%%%%%%%%%%%%%%%%%%%%%%%%%%%%%%%%%%%%%%%%%%%%%%%%%%%%%%%%%%%%%%%%%

The following result is a generalization of \cite[Theorem 1.2]{5} to operator-valued dual frames as well as fusion frame duals with different proof.

\begin{theorem}\label{dual2}
Suppose that ${\mathcal A}=\{A_i\}_{i\in I}\subset B({\mathcal H},{\mathcal K})$ is an operator-valued frame and that $({\mathcal W},\omega)$ and $({\mathcal W}',\omega')$ are fusion frames for $\mathcal H$. The following assertions are hold.
\begin{enumerate}
\item The closure of the union of all sets ${\rm ran}(T_{\widetilde{A}(L)})$ is
${\mathfrak{K}}$, where $\widetilde{A}(L)$ runs through all dual operator-valued frame of ${\mathcal A}$.
\item If ${\mathcal B}=\{B_i\}_{i\in I}\subset B({\mathcal H},{\mathcal K})$ is an operator-valued Bessel sequence such that $T^*_{\mathcal B}T_{{\widetilde{A}(L)}}=0$ for every $L\in {\mathcal L}_{\mathcal A}$, then $\mathcal B$ is the null-sequence.
\item If ${\rm Inv}_l(T_{{\mathcal W},\omega})$ denotes the set of all bounded left inverse of $T_{{\mathcal W},\omega}$ and bar refers to the norm closure, then
    $$\overline{\bigcup_{U\in {\rm Inv}_l(T_{{\mathcal W},\omega})}{\rm ran}(U^*)}=\mathfrak{H}.$$
\item If every fusion frame dual $({\mathcal V},\upsilon)$ of $({\mathcal W},\omega)$ is a fusion frame dual of $({\mathcal W}',\omega')$, then
$$\omega_i P_{{\mathcal W}_i}=\omega'_i P_{{\mathcal W}'_i}\quad\quad\quad(i\in I).$$
\end{enumerate}
\end{theorem}
\begin{proof}
We prove the assertions (1) and (2); since, with the aid of Proposition \ref{p8}, the proof of (3) and (4) are similar.

We first prove (1), which is the essential part of the theorem.
To this end, suppose that we had an element $\{y_i\}_{i\in I}$ in $\mathfrak{K}$ not belonging to
$\overline{\bigcup_{L\in {\mathcal L}_{\mathcal A}}{\rm ran}(T_{\widetilde{A}(L)})}$. Then, by the Hahn-Banach and Riesz Representation Theorems, we would have a
$\{z_i\}_{i\in I}$ in $\mathfrak{K}$ such that $\big<\{z_i\}_{i\in I},\{y_i\}_{i\in I}\big>\neq 0$ and
$$\{z_i\}_{i\in I}\in {\Big(}\;\overline{\bigcup_{L\in {\mathcal L}_{\mathcal A}}{\rm ran}(T_{\widetilde{A}(L)})}\;{\Big)}^\perp.$$
In particular, $\{z_i\}_{i\in I}\perp{\rm ran}(T_{\widetilde{A}(L)})$ for all $L\in {\mathcal L}_{\mathcal A}$.
However, this is
not possible; This is because of, if we define $\vartheta_e:{\mathcal H}\rightarrow{\mathfrak{K}}$ by $\vartheta_e(f):=\big<f,e\big>\{y_i\}_{i\in I}$, where
$e\in {\mathcal H}$ is such that $\|e\|=1$. Then $\vartheta_e$
is in $B({\mathcal H},{\mathfrak{K}})$ and $\vartheta_e(e)=\{y_i\}_{i\in I}$.
Moreover, for $\widetilde{A}(0)$ and $\widetilde{A}(P_{\ker(T_{\mathcal A}^*)}\vartheta_e)$, we have
$$T_{\widetilde{A}(P_{\ker(T_{\mathcal A}^*)}\vartheta_e)}^*\{z_i\}_{i\in I}=0=T_{\widetilde{A}(0)}^*\{z_i\}_{i\in I}=S_{\mathcal A}^{-1}T_{\mathcal A}^*\{z_i\}_{i\in I}.$$
It follows that $T_{\mathcal A}^*\{z_i\}_{i\in I}=0$ and thus
\begin{align*}
\big<\{z_i\}_{i\in I},\{y_i\}_{i\in I}\big>&=\big<\{z_i\}_{i\in I},\vartheta_e(e)\big>\\
&=\big<\vartheta_e^*P_{\ker(T_{\mathcal A}^*)}\{z_i\}_{i\in I},e\big>\\
&=\big<(T_{\widetilde{A}(0)}^*+\vartheta_e^*P_{\ker(T_{\mathcal A}^*)})\{z_i\}_{i\in I},e\big>\\
&=\big<T_{\widetilde{A}(P_{\ker(T_{\mathcal A}^*)}\vartheta_e)}^*\{z_i\}_{i\in I},e\big>\\
&=0.
\end{align*}
This contradiction proves that (1) holds.

To prove (2), we note that the hypothesis $T^*_{\mathcal B}T_{{\widetilde{A}(L)}}=0$ for every $L\in {\mathcal L}_{\mathcal A}$, together with the density of $\cup_{L\in {\mathcal L}_{\mathcal A}}{\rm ran}(T_{\widetilde{A}(L)})$ in $\mathfrak{K}$ imply that $\ker(T^*_{\mathcal B})={\mathfrak{K}}$. This says that ${\rm ran}(T_{\mathcal B})=\{0\}$ and thus $B_i=0$ for all $i\in I$.
\end{proof}

%%%%%%%%%%%%%%%%%%%%%%%%%%%%%%%%%%%%%%%%%%%%%%%%%%%%%%%%%%%%%%%%%%%%%%%%%%%%%%%%%%%%%%%%%%%%%%%%%%%%%%%%%%%

Since their introduction in 2007 (see \cite{balaz3}), ordinary Bessel multipliers, as a generalization of Gabor multipliers \cite{feichmulti1}, have been extensively
generalized and studied, see for example \cite{ar,5,j,jc,msh}. The reader will remark that invertible Bessel multipliers is a proper generalization of duality notion in Hilbert spaces which permits us to have different reconstruction strategies.
However, there has only been one approach yet for studying the invertibility of fusion frame multipliers \cite{msh}.
In the following, we want to study a new concept of
Bessel fusion multipliers in Hilbert spaces, which is a slightly modified version of \cite{ar,msh}. In particular, we
show that with this definition in many cases Bessel fusion multipliers behave
similar to ordinary Bessel multipliers, see Remark \ref{tojih} below.

Our definition of the term ``Bessel fusion multiplier'' is given in the following.
Before that, let us note that if we associate to a sequence $m=\{m_i\}_{i\in I}$ and ${\mathcal R}=\{R_i\}_{i\in I}\in\ell^\infty(I,B({\mathcal H}))$ the following bounded operators:
$${\mathcal M}_m:{\mathfrak{H}}\rightarrow{\mathfrak{H}};\quad \{x_i\}_{i\in I}\mapsto\{m_ix_i\}_{i\in I},$$
and
$$D_{\mathcal R}:{\mathfrak{H}}\rightarrow{\mathfrak{H}};\quad \{x_i\}_{i\in I}\mapsto\{R_ix_i\}_{i\in I},$$
then
$${\bf M}_{m{\mathcal R},{\mathcal V},{\mathcal W}}:=T^*_{{\mathcal V},\upsilon}{\mathcal M}_mD_{\mathcal R}T_{{\mathcal W},\omega}=T^*_{{\mathcal V},\upsilon}D_{m\mathcal R}T_{{\mathcal W},\omega}$$
is well-defined and $$\|{\bf M}_{m{\mathcal R},{\mathcal V},{\mathcal W}}\|\leq\sqrt{\beta_{{\mathcal V}_\upsilon}\beta_{{\mathcal W}_\omega}}\|m\|_\infty\|{\mathcal R}\|_\infty,$$
where here and in the sequel $D_{m\mathcal R}={\mathcal M}_m\circ D_{\mathcal R}$ and $m{\mathcal R}$ refers to the sequence $\{m_iR_i\}_{i\in I}$.

\begin{definition}
Let $({\mathcal W},\omega)$ and $({\mathcal V},\upsilon)$ be two Bessel fusion sequences for $\mathcal H$. If $m\in\ell^\infty(I)$ and ${\mathcal R}\in\ell^\infty(I,B({\mathcal H}))$, then the operator ${\bf M}_{m{\mathcal R},{\mathcal V},{\mathcal W}}$
is called a Bessel fusion multiplier with symbol $(m,{\mathcal R})$ or simply $(m,{\mathcal R})$-Bessel fusion multiplier.
\end{definition}

%%%%%%%%%%%%%%%%%%%%%%%%%%%%%%%%%%%%%%%%%%%%%%%%%%%%%%%%%%%%%%%%%%%%%%%%%%%%%%%%%%%%%%%%

That $(m,{\mathcal R})$-Bessel fusion multipliers generalize the notion of fusion frame duals follows from the facts that for each $i\in I$ we have
$$P_{V_i}R_iP_{W_i}=P_{V_i}P_{V_i}R_iP_{W_i}P_{W_i},$$
and
$$W_i^\perp\subseteq \ker(P_{V_i}R_iP_{W_i})\quad\quad{\hbox{and}}\quad\quad {\rm ran}(P_{V_i}R_iP_{W_i})\subseteq V_i.$$
It also generalized the notion of ordinary Bessel multipliers. To see this, suppose that $\Phi=\{\varphi_i\}_{i\in I}$ and $\Psi=\{\psi_i\}_{i\in I}$ are two Bessel sequences in ${\mathcal H}$. If,
for every $i\in I$, $W_i:={\rm span}\{\psi_i\}$, $V_i:={\rm span}\{\varphi_i\}$, $\omega_i:=\|\psi_i\|$, $\upsilon_i:=\|\varphi_i\|$, then $({\mathcal W},\omega)$ and
$({\mathcal V},\upsilon)$ are Bessel fusion sequence which is called the Bessel fusion sequence related to $\Phi$ and $\Psi$, respectively.
In particular, if we define
$R_i:{\mathcal H}\rightarrow{\mathcal H}$ by
\[ R_ix=\left\{%
\begin{array}{ccc}
\hspace{-1.9cm}\vspace{2mm}0 &
\quad\quad\quad\quad\hbox{if}~ \varphi_i=0~{\hbox{or}}~\psi_i=0\\
\big<x,\frac{\psi_i}{\|\psi_i\|}\big>\frac{\varphi_i}{\|\varphi_i\|} & \hbox{otherwise}\\
\end{array}%
\right.\]
then $\{R_i\}_{i\in I}$ is in $\ell^\infty(I,B({\mathcal H}))$. Particularly, for each symbol $m$, the $\mathcal R$-Bessel fusion multiplier ${\bf M}_{m{\mathcal R},{\mathcal V},{\mathcal W}}$
is equal to the ordinary Bessel multiplier ${\bf M}_{m,\Phi,\Psi}(x):=\sum_{i\in I}\big<x,\psi_i\big>\varphi_i$ ($x\in{\mathcal H}$). This is because of,
for each $x\in{\mathcal H}$ we have
\begin{align*}
\sum_{i\in I}m_i\omega_i\upsilon_iP_{V_i}R_iP_{W_i}x&=\sum_{i\in I,\psi_i\neq0,\varphi_i\neq0}m_i\omega_i\upsilon_i\big<x,\frac{\psi_i}
{\|\psi_i\|}\big>\frac{\varphi_i}{\|\varphi_i\|}\\
&=\sum_{i\in I}m_i\big<x,\psi_i\big>\varphi_i.
\end{align*}

%%%%%%%%%%%%%%%%%%%%%%%%%%%%%%%%%%%%%%%%%%%%%%%%%%%%%%%%%%%%%%%%%%%%%%%%%%%%%%%%%%%%%%

The following note is a very useful tool in our present arguments.

\begin{note}
Suppose that $m\in\ell^\infty(I)$ and that ${\mathcal R}\in\ell^\infty(I,B({\mathcal H}))$. Then it is easy to see that ${\mathcal M}_m^*$ and $D_{\mathcal R}^*$, the adjoint of ${\mathcal M}_m$ and $D_{\mathcal R}$, respectively, are given by
$${\mathcal M}_m^*={\mathcal M}_{\overline{m}}:{\mathfrak{H}}\rightarrow{\mathfrak{H}};\quad \{x_i\}_{i\in I}\mapsto\{\overline{m_i}\;x_i\}_{i\in I},$$
and
$$D_{\mathcal R}^*:{\mathfrak{H}}\rightarrow{\mathfrak{H}};\quad \{x_i\}_{i\in I}\mapsto\{R_i^*x_i\}_{i\in I},$$
where $\overline{m}$ denote the sequence $\{\overline{m_i}\}_{i\in I}\in\ell^\infty(I)$ and for each $i\in I$ the letter $\overline{m_i}$ refers to the complex conjugate of $m_i$.
Particularly, if ${\mathcal C}(m,{\mathcal R})$ refers to the following hypothesis:
$$\exists\gamma, \delta>0\quad{\hbox{such~that}}\quad\gamma\|x\|\leq\|\overline{m_j}R_j^*x\|
    \leq\delta\|x\|\quad\quad\forall x\in{\mathcal H}~~ {\hbox{and}}~~\forall j\in I,$$
then
\begin{enumerate}
\item for each $j\in I$, the operator $m_jR_j$ is invertible and the sequence ${(m{\mathcal R})^{-1}}:=\{(m_iR_i)^{-1}\}_{i\in I}$ is in $\ell^\infty(I,B({\mathcal H}))$. It follows that the bounded operator
$$D_{(m\mathcal R)^{-1}}:{\mathfrak{H}}\rightarrow{\mathfrak{H}};\quad \{x_i\}_{i\in I}\mapsto\{(m_iR_i)^{-1}x_i\}_{i\in I},$$
is the inverse of $D_{m\mathcal R}$;
\item the sequence $m\in\ell^\infty(I)$ is semi-normalized, that is,
$$\frac{\gamma}{\|{\mathcal R}\|_\infty}<\inf_{i\in I}|m_i|\leq\sup_{i\in I}|m_i|<\infty;$$
\end{enumerate}
\end{note}

%%%%%%%%%%%%%%%%%%%%%%%%%%%%%%%%%%%%%%%%%%%%%%%%%%%%%%%%%%%%%%%%%%%%%%%%%%%%%%%%%%%%%

The following proposition gives some necessary conditions for invertibility of
$(m,{\mathcal R})$-Bessel fusion multipliers. As usual, the excess of Bessel fusion sequence $({\mathcal W},\omega)$ is defined as $e({\mathcal W},\omega):=\dim{\Big(}\ker(T_{{\mathcal W},\omega}^*){\Big)}$.

\begin{proposition}\label{gelar}
Let $({\mathcal W},\omega)$ and $({\mathcal V},\upsilon)$ be two Bessel fusion sequences for $\mathcal H$ and let $m\in\ell^\infty(I)$ and ${\mathcal R}\in\ell^\infty(I,B({\mathcal H}))$. Assume that the Bessel fusion multiplier ${\bf M}_{m{\mathcal R},{\mathcal V},{\mathcal W}}$ is invertible. Hence,
\begin{enumerate}
\item The fusion sequences $({\mathcal W},\omega)$, $({\mathcal V},\upsilon)$, $({\mathcal W},m^{av}\omega)$ and $({\mathcal V},m^{av}\upsilon)$ are fusion frames for $\mathcal H$, where $m^{av}:=\{|m_i|\}_{i\in I}$, $m^{av}\omega:=\{|m_i|\omega_i\}_{i\in I}$ and $m^{av}\upsilon:=\{|m_i|\upsilon_i\}_{i\in I}$.
    \item If $\inf_{i\in I}|m_i|>0$, then $$e({\mathcal W},\omega)=e({\mathcal W},m^{av}\omega)\quad\quad\quad{\hbox{and}}\quad\quad\quad e({\mathcal V},\upsilon)=e({\mathcal V},m^{av}\upsilon).$$
    \item If ${\mathcal C}(m,{\mathcal R})$ holds, then $({\mathcal W},\omega)$ and $({\mathcal V},\upsilon)$ have the same excess.
\end{enumerate}
\end{proposition}
\begin{proof}
(1) For briefness, we only show that $({\mathcal W},m\omega)$ satisfies the lower fusion frame condition. To this end, suppose that $x$ is an arbitrary element of $\mathcal H$. We observe that
\begin{align*}
\|{\bf M}_{m{\mathcal R},{\mathcal V},{\mathcal W}}x\|^2&=\sup_{\|y\|=1}|\big<T^*_{{\mathcal V},\upsilon}{\mathcal M}_mD_{\mathcal R}T_{{\mathcal W},\omega}x,y\big>|^2\\
&=\sup_{\|y\|=1}|\big<\sum_{i\in I}m_i\upsilon_i\omega_iP_{V_i}R_iP_{W_i}x,y\big>|^2\\
&\leq\sup_{\|y\|=1}\sum_{i\in I}|m_i|^2\omega_i^2\|P_{W_i}x\|^2\sum_{i\in I}\upsilon_i^2\|R_i^*P_{V_i}y\|^2\\
&\leq B_{{\mathcal V}_\upsilon}\|{\mathcal R}\|_\infty \sum_{i\in I}|m_i|^2\omega_i^2\|P_{W_i}x\|^2.
\end{align*}
We now invoke the inequality $\|x\|^2\leq\|{\bf M}_{m{\mathcal R},{\mathcal V},{\mathcal W}}^{-1}\|^2\|{\bf M}_{m{\mathcal R},{\mathcal V},{\mathcal W}}x\|^2$ to conclude that
$$A_{{\mathcal W}_{m^{av}\omega}}:=1/B_{{\mathcal V}_\upsilon}\|{\mathcal R}\|_\infty\|{\bf M}_{m{\mathcal R},{\mathcal V},{\mathcal W}}^{-1}\|^2$$ is
a lower fusion frame bound for $({\mathcal W},m^{av}\omega)$.

(2) We prove the first equality; the proof of the second is similar. To this end, we note that $T_{{\mathcal W},m\omega}^*=T_{{\mathcal W},\omega}^*{\mathcal M}_{m^{av}}$, where
$${\mathcal M}_{m^{av}}:{\mathfrak{H}}\rightarrow{\mathfrak{H}};\quad \{x_i\}_{i\in I}\mapsto\{|m_i|x_i\}_{i\in I}.$$
Moreover, since $\inf_{i\in I}|m_i|>0$, the bounded linear operator ${\mathcal M}_{1/m^{av}}\{x_i\}_{i\in I}:=\{\frac{1}{|m_i|}x_i\}_{i\in I}$ is the inverse of ${\mathcal M}_{m^{av}}$. It follows that
$$\ker(T_{{\mathcal W},m\omega}^*)=\ker(T_{{\mathcal W},\omega}^*{\mathcal M}_{m^{av}})={\mathcal M}_{1/m^{av}}(\ker(T_{{\mathcal W},\omega}^*)),$$
and this completes the proof of the assertion (2).

(3) Denote ${\bf M}:={\bf M}_{m{\mathcal R},{\mathcal V},{\mathcal W}}$. If we define $U:{\mathfrak{H}}\rightarrow{\mathfrak{H}}$ by $$U\{x_i\}_{i\in I}:={\Big(}Id_{\mathfrak{H}}-T_{{\mathcal W},\omega}{\bf M}^{-1}T_{{\mathcal V},\upsilon}^*D_{m\mathcal R}{\Big)}\{x_i\}_{i\in I},$$ then it is not hard to check that the nullity of the operator $T^*_{{\mathcal V},\upsilon}D_{m\mathcal R}(Id_{\mathfrak{H}}-T_{{\mathcal W},\omega}{\bf M}^{-1}T^*_{{\mathcal V},\upsilon}D_{m\mathcal R})$ implies that
$${\rm ran}(U)=\ker(T_{{\mathcal V},\upsilon}^*D_{m\mathcal R}).$$
On the other hand, using the nullity of the operator $UT_{{\mathcal W},\omega}{\bf M}^{-1}$ and the equality $${\rm ran}(T_{{\mathcal W},\omega}{\bf M}^{-1})\oplus\ker((T_{{\mathcal W},\omega}{\bf M}^{-1})^*)={\mathfrak{H}},$$
one can conclude that
$${\rm ran}(U)=U(\ker((T_{{\mathcal W},\omega}{\bf M}^{-1})^*)).$$
It follows that
$$\ker(T_{{\mathcal V},\upsilon}^*D_{m\mathcal R})=U(\ker((T_{{\mathcal W},\omega}{\bf M}^{-1})^*)).$$
Hence, we have
\begin{align*}
\dim{\Big(}\ker(T_{{\mathcal V},\upsilon}^*)
{\Big)}&=
\dim{\Big(}D_{(m{\mathcal R})^{-1}}(\ker(T_{{\mathcal V},\upsilon}^*))
{\Big)}\\&=
\dim{\Big(}\ker(T_{{\mathcal V},\upsilon}^*D_{m{\mathcal R}}){\Big)}\\
&=\dim{\Big(}U(\ker((T_{{\mathcal W},\omega}{\bf M}^{-1})^*)){\Big)}\\
&\leq\dim{\Big(}\ker((T_{{\mathcal W},\omega}{\bf M}^{-1})^*){\Big)}\\
&=\dim{\Big(}\ker(({\bf M}^{-1})^*T_{{\mathcal W},\omega}^*){\Big)}\\
&=\dim{\Big(}\ker(T_{{\mathcal W},\omega}^*){\Big)}.
\end{align*}
Similarly, one can show that
$$\dim{\Big(}\ker(T_{{\mathcal W},\omega}^*){\Big)}\leq \dim{\Big(}\ker(T_{{\mathcal V},\upsilon}^*)
{\Big)}.$$
We have now completed the proof of the assertion (3).
\end{proof}

%%%%%%%%%%%%%%%%%%%%%%%%%%%%%%%%%%%%%%%%%%%%%%%%%%%%%%%%%%%%%%%%%%%%%%%%%%%%%%%%%%%%%%%

Recall from \cite{SB2} that
in the case where $\Phi=\{\varphi_i\}_{i\in I}$ and $\Psi=\{\psi_i\}_{i\in I}$ are two Riesz bases for $\mathcal H$ and
$m$ is chosen so that
$$0<\inf_{i\in I}|m_i|\leq\sup_{i\in I}|m_i|<\infty,$$
that is, $m$ is semi-normalized, then a Riesz multiplier ${\bf M}_{m,\Phi,\Psi}=T_{\Phi}^*{\mathcal M}_{m}T_{\Psi}$ is automatically invertible without any extra assumptions on Riesz bases $\Phi$, $\Psi$. This result is
also true for Riesz fusion multiplier ${\bf M}_{m{\mathcal R},{\mathcal V},{\mathcal W}}$, whenever $m$ and $\mathcal R$ are chosen so that the hypothesis ${\mathcal C}(m,{\mathcal R})$ is valid; This is because of, in this case the operator
${\bf M}_{m{\mathcal R},{\mathcal V},{\mathcal W}}$ is the composition of three invertible operators $T^*_{{\mathcal V},\upsilon}$, $D_{m{\mathcal R}}$ and $T_{{\mathcal W},\omega}$.
This stands in contrast to the notions of Bessel fusion multipliers in the sense of \cite{ar,msh}, see Remark \ref{tojih} below.

The following theorem completely characterize invertible $(m,{\mathcal R})$-Riesz fusion multipliers. It is an analogous of the main result of \cite{SB2} for ordinary Riesz multipliers. Note that our approach is different from that of \cite[Theorem 5.1]{SB2}.
In particular, it shows that having a proper freedom in the choice of the
symbol $(m,{\mathcal R})$ permits us to have infinitely many invertible Riesz fusion multipliers as well as different reconstruction strategies.

\begin{theorem}\label{parsa1257}
Suppose that $({\mathcal W},\omega)$ is a fusion Riesz basis in $\mathcal H$.
Then the following assertions hold.
\begin{enumerate}
\item If $({\mathcal V},\upsilon)$ is a fusion Riesz basis, then the Bessel fusion multiplier ${\bf M}_{m{\mathcal R},{\mathcal V},{\mathcal W}}$ is invertible on $\mathcal H$ if and only if ${\mathcal C}(m,{\mathcal R})$ holds.
    \item If ${{\mathcal C}(m,{\mathcal R})}$ holds, then the Bessel fusion multiplier ${\bf M}_{m{\mathcal R},{\mathcal V},{\mathcal W}}$ is invertible on $\mathcal H$ if and only if $({\mathcal V},\upsilon)$ is a fusion Riesz basis.
\end{enumerate}
\end{theorem}
\begin{proof}
(1) If ${\mathcal C}(m,{\mathcal R})$ holds, then part (3) of Proposition \ref{gelar} implies that
$$\dim(\ker(T_{{\mathcal W},\omega}^*))=\dim(\ker(T_{{\mathcal V},\upsilon}^*)).$$
From this, we can deduce that $\ker(T_{{\mathcal V},\upsilon}^*)$ is isomorphic to
$\ker(T_{{\mathcal W},\omega}^*)=\{0\}$. Hence, the operator $T_{{\mathcal V},\upsilon}^*$ is injective and thus $({\mathcal V},\upsilon)$ is a Riesz basis.

Conversely, suppose that ${\bf M}_{m{\mathcal R},{\mathcal V},{\mathcal W}}$ is invertible. This together with the invertibility of the operator
$T_{{\mathcal W},\omega}$ imply that $T^*_{{\mathcal V},\upsilon}{\mathcal M}_mD_{\mathcal R}$ is a bounded invertible operator. Hence, the Open Mapping Theorem guarantees that
there exist positive constants $\alpha$ and $\beta$ such that
\begin{align}\label{p10}
\alpha\|x\|\leq\|D_{\mathcal R}^*{\mathcal M}_{\overline{m}}T_{{\mathcal V},\upsilon}x\|\leq\beta\|x\|\quad\quad\quad\quad\quad(x\in{\mathcal H}).
\end{align}
Moreover, for each $x\in{\mathcal H}$ and $j\in I$, we observe that
\begin{align*}
\|\overline{m_j}R_j^*x\|&=\|D_{\mathcal R}^*{\mathcal M}_{\overline{m}}\{\delta_{i,j}x\}_{i\in I}\|\\
&=\|D_{\mathcal R}^*{\mathcal M}_{\overline{m}}T_{{\mathcal V},\upsilon}T_{{\mathcal V},\upsilon}^{-1}\{\delta_{i,j}x\}_{i\in I}\|,
\end{align*}
where $\delta_{i,j}$ refers to the Kronecker delta.
From this, by inequalities \ref{p10}, we deduce that
\begin{align*}
\frac{\alpha}{\|T_{{\mathcal V},\upsilon}\|}\|x\|&=\frac{\alpha}{\|T_{{\mathcal V},\upsilon}\|}\|\{\delta_{i,j}x\}_{i\in I}\|\\
&\leq\alpha\|T_{{\mathcal V},\upsilon}^{-1}\{\delta_{i,j}x\}_{i\in I}\|\\
&\leq\|D_{\mathcal R}^*{\mathcal M}_{\overline{m}}T_{{\mathcal V},\upsilon}T_{{\mathcal V},\upsilon}^{-1}\{\delta_{i,j}x\}_{i\in I}\|\\
&\leq\beta\|T_{{\mathcal V},\upsilon}^{-1}\|\|x\|,
\end{align*}
for all $x\in{\mathcal H}$, and this completes the proof of this part.

(2) In light of part (3) of Proposition \ref{gelar} and part (1) the proof is trivial and so the details are omitted.
\end{proof}

%%%%%%%%%%%%%%%%%%%%%%%%%%%%%%%%%%%%%%%%%%%%%%%%%%%%%%%%%%%%%%%%%%%%%%%%%%%%%%%%%%%%%%%

For the formulation of the following statements we ask the reader to recall from Section 2 that $T_{{\mathcal W},\omega}$ is an operator from $\mathcal H$ into $\mathfrak{H}$ whereas the operator $T_{\mathcal W}$ is an operator from $\mathcal H$ into ${\mathcal K}_{\mathcal W}$.
Also, it is worth noticing that, the following remark illustrates that
\begin{itemize}
\item the invertibility of a Riesz fusion multipliers in the sense of \cite{ar,msh} does not depends only on its symbol whereas the invertibility of $(m,{\mathcal R})$-Riesz fusion multipliers depend only on the chosen of their symbols;
\item the set of all invertible Riesz fusion multipliers in the sense of \cite{ar,msh} is a proper subset of the set of all invertible $(m,{\mathcal R})$-Riesz fusion multipliers.
\end{itemize}

\begin{remark}\label{tojih}
Suppose that $({\mathcal W},\omega)$ and $({\mathcal V},\upsilon)$ are two Bessel fusion sequences and that $m\in\ell^\infty(I)$.
\begin{enumerate}
\item A notion of Bessel fusion multipliers was defined in \cite{ar} as follows:
$${\bf M}_{m,{\mathcal V},{\mathcal W}}:{\mathcal H}\rightarrow{\mathcal H};\quad x\mapsto T^*_{{\mathcal V}}{\mathcal S}_mPT_{{\mathcal W}}x,$$
where ${\mathcal S}_m:{\mathcal K}_{\mathcal V}\rightarrow{\mathcal K}_{\mathcal V}$ is defined by ${\mathcal S}_m\{x_i\}_{i\in I}:=\{m_ix_i\}_{i\in I}$ and $$P:{\mathcal K}_{\mathcal W}\rightarrow{\mathcal K}_{\mathcal V};\quad P\{x_i\}_{i\in I}:=\{P_{V_i}x_i\}_{i\in I}=\{P_{V_i}|_{W_i}x_i\}_{i\in I}.$$
The reader will remark that this notion of Bessel multiplier can be considered as a special weighted version of Eq. (\ref{q-parsa}).

Now, suppose that $({\mathcal W},\omega)$ and $({\mathcal V},\upsilon)$ are Riesz fusion bases, that is, the operators $T_{\mathcal W}^*$ and $T_{\mathcal V}^*$ are injective. If  ${\bf M}_{m,{\mathcal V},{\mathcal W}}$ is invertible, then by a method similar to that of Theorem \ref{parsa1257} one can show that there exist
positive number $\gamma$ such that
$$\quad\quad\gamma\|x\|\leq\|\overline{m_i}(P_{V_i}|_{W_i})^*x\|=\|\overline{m_i}P_{W_i}|_{V_i}x\|
\quad\quad\quad\quad\quad(x\in V_i, i\in I).$$
Hence, if $m$ is semi-normalized, then we can deduce that for each $i\in I$, the operators $P_{W_i}|_{V_i}\in B(V_i,W_i)$ are bounded below and thus they are isomorphisms. We now invoke Proposition 3.3 of \cite{invar} to conclude that the Riesz fusion multiplier ${\bf M}_{m,{\mathcal V},{\mathcal W}}$ is invertible if and only if ${\mathcal H}=W_i\oplus V_i^\perp$ for all $i\in I$, which seems to be of hardly any use in applications. Hence, there exist infinitely many non-invertible Riesz fusion multipliers with semi-normalized symbol in the sense of \cite{ar}.
\item Another notion of Bessel fusion multipliers was defined in \cite{msh} as follows:
$${\bf M}_{m,{\mathcal V},{\mathcal W}}:{\mathcal H}\rightarrow{\mathcal H};\quad x\mapsto T^*_{{\mathcal V}}\phi_{{\mathcal V}{\mathcal W}}T_{{\mathcal W}}x,$$
where $\phi_{{\mathcal V}{\mathcal W}}:{\mathcal K}_{\mathcal W}\rightarrow{\mathcal K}_{\mathcal V}$ is defined by $\phi_{{\mathcal V}{\mathcal W}}\{x_i\}_{i\in I}:=\{m_iP_{V_i}S_{\mathcal W}^{-1}x_i\}_{i\in I}$.
The reader will remark that this notion of Bessel multiplier can be considered as a weighted version of Eq. (\ref{p2}).

With this definition of Bessel fusion multipliers, there exist infinitely many non-invertible Riesz fusion multipliers with semi-normalized symbol.
This is because of, if $\mathcal V$ is a G$\check{a}$vru$\c{t}$a-dual of $\mathcal W$, then, in general, $\mathcal W$ is not a G$\check{a}$vru$\c{t}$a-dual of $\mathcal V$, by what was
mentioned in Section 1. Hence, a Riesz fusion multiplier in the sense of \cite{msh} may have a right inverse which is not its left inverse, see \cite[page 7]{msh}.
On the other hand, an argument similar to the proof of Theorem \ref{parsa1257} shows that, if ${\bf M}_{m,{\mathcal V},{\mathcal W}}$ is an invertible Riesz fusion multiplier, then there exist $\gamma>0$ such that
$$\gamma\|x\|\leq\|m_iP_{V_i}S_{\mathcal W}^{-1}x\|
\quad\quad\quad\quad\quad(x\in W_i, i\in I),$$
which may not be satisfied for all $V_i$ ($i\in I$).
Hence, it seems to be of hardly any use in applications.
\end{enumerate}
\end{remark}

%%%%%%%%%%%%%%%%%%%%%%%%%%%%%%%%%%%%%%%%%%%%%%%%%%%%%%%%%%%%%%%%%%%%%%%%%%%%%%%%%%%%%%%%%

Recall from \cite{5} that if ${\bf M}_{m,\Phi,\Psi}$ is invertible, where $\Phi$ and $\Psi$ are frames for $\mathcal H$ and $m$ is semi-normalized, then  $\Psi^\dagger:=\{{\bf M}^{-1}_{m,\Phi,\Psi}(m_i\varphi_i)\}_{i\in I}$ is the unique dual frame of $\Psi$ such that
\begin{eqnarray*}
{\bf M}^{-1}_{m,\Phi,\Psi}={\bf M}_{1/m,{\Psi}^\dagger,{\Phi}^d}\quad\quad\quad\forall~{\hbox{dual frames}}~\Phi^d~{\hbox{of}}~\Phi.
\end{eqnarray*}
But, there does not seem to be an easy way to show that
the inverse of an invertible fusion frame
multiplier can be represented as a fusion multiplier with the reciprocal symbol and fusion frame duals of the given ones. However, in light of Proposition \ref{p8} and Theorem \ref{dual2} we have the following result.

\begin{theorem}
Let $({\mathcal W},\omega)$ and $({\mathcal V},\upsilon)$ be two fusion frames for $\mathcal H$. Suppose that for $m\in\ell^\infty(I)$ and ${\mathcal R}\in\ell^\infty(I,B({\mathcal H}))$ hypothesis ${\mathcal C}(m,{\mathcal R})$ are valid.
If ${\bf M}:={\bf M}_{m{\mathcal R},{\mathcal V},{\mathcal W}}$ is invertible, then
$Q^\dag:=\{Q_i^\dag:=\omega_iP_{W_i}S^{-1}_{{\mathcal W},\omega}+\overline{m_i}L_i\}_{i\in I}$ is the
unique operator-valued dual frame of ${\mathcal A}_{\mathcal W}$ such that
$${\bf M}^{-1}=T_{Q^\dag}^*D_{(mR)^{-1}}T_{{{\mathcal A}_{\mathcal V}}^{opd}}$$
for all operator-valued dual frame ${{\mathcal A}_{\mathcal V}}^{opd}$ of ${\mathcal A}_{\mathcal V}$, where, for each $i\in I$, $L_i$ is the unique solution of the operator equation
\begin{equation}\label{parsaksh}
{\bf M}L_i^*=\upsilon_iP_{V_i}R_i-\frac{\omega_i}{m_i}{\bf M}S_{\mathcal W}^{-1}P_{W_i},
\end{equation}
which can be obtained by Douglas's Theorem \cite{dou} and its proof.
\end{theorem}
\begin{proof}
By (\ref{parsaksh}), for each $i\in I$, we have
$$L_i=\upsilon_iR_i^*P_{V_i}({\bf M}^*)^{-1}-\frac{\omega_i}{\overline{m_i}}P_{W_i}S_{\mathcal W}^{-1}.$$
Observe that ${\mathcal L}:=\{L_i\}_{i\in I}$ is an operator-valued Bessel sequence for which
\begin{align*}
T_{Q^\dag}^*\{x_i\}_{i\in I}&=\sum_{i\in I}\omega_iS_{\mathcal W}^{-1}P_{W_i}x_i+m_i\upsilon_i{\bf M}^{-1}P_{V_i}R_ix_i-\omega_iS_{\mathcal W}^{-1}P_{W_i}x_i\\
&=\sum_{i\in I}m_i\upsilon_i{\bf M}^{-1}P_{V_i}R_ix_i\\
&={\bf M}^{-1}T_{{\mathcal V},\upsilon}^*D_{m{\mathcal R}}\{x_i\}_{i\in I},
\end{align*}
for all $\{x_i\}_{i\in I}\in\mathfrak{H}$.
From this, we deduced that
$$T_{Q^\dag}^*T_{{\mathcal W},\omega}={\bf M}^{-1}T_{{\mathcal V},\upsilon}^*D_{m{\mathcal R}}T_{{\mathcal W},\omega}={\bf M}^{-1}{\bf M}=Id_{\mathcal H}.$$
It follows that $Q^\dag$ is an operator-valued dual frame of ${\mathcal A}_{\mathcal W}$.
Moreover, if ${{\mathcal A}_{\mathcal V}}^{opd}$ is an arbitrary operator-valued dual frame  of ${\mathcal A}_{\mathcal V}$, then, on the one hand, we have
\begin{align*}
{\bf M}(T_{Q^\dag}^*D_{m{\mathcal R}}T_{{{\mathcal A}_{\mathcal V}}^{opd}})={\bf M}({\bf M}^{-1}T_{{\mathcal V},\upsilon}^*D_{m{\mathcal R}}D_{(m{\mathcal R})^{-1}}T_{{{\mathcal A}_{\mathcal V}}^{opd}})=Id_{\mathcal H},
\end{align*}
and on the other hand
$$(T_{Q^\dag}^*D_{m{\mathcal R}}T_{{{\mathcal A}_{\mathcal V}}^{opd}}){\bf M}=T_{Q^\dag}^*T_{{\mathcal W},\omega}=Id_{\mathcal H}.$$
Finally, we note that the uniqueness of $Q^\dag$ follows from Theorem \ref{dual2}
and this completes the proof of the theorem.
\end{proof}

%%%%%%%%%%%%%%%%%%%%%%%%%%%%%%%%%%%%%%%%%%%%%%%%%%%%%%%%%%%%%%%%%%%%%%%%%%%%%%%%%%%%%%%%5

For the formulation of the next result which describe the invertibility of $(m,{\mathcal R})$-Bessel frame multipliers in terms of local frames, we need to recall the notion of local frames as well as local dual frames. Toward this end, suppose that $({\mathcal W},\omega)$
is a fusion frame for $\mathcal H$. If $\Phi_i=\{\varphi_{i,j}\}_{j\in J_i}$ ($i\in I$) is a frame for $W_i$ with frame bounds $\alpha_i$ and $\beta_i$, respectively, such that $0<\alpha=\inf_{i\in I}\alpha_i\leq\beta=\sup_{i\in I}\beta_i<\infty$, then the sequence $\{\Phi_i\}_{i\in I}$ is called local frames of $({\mathcal W},\omega)$. Also, if $\widetilde{\Phi_i}(0)=\{\widetilde{\varphi_{i,j}}(0)\}_{j\in J_i}$ is the canonical dual frame of $\Phi_i$ in $W_i$, then we call $\{\widetilde{\Phi_i}(0)\}_{i\in I}$ the local canonical dual frame of $\{\Phi_i\}_{i\in I}$.

\begin{theorem}
Let $({\mathcal W},\omega)$ and $({\mathcal V},\upsilon)$ be two fusion frames for $\mathcal H$. Suppose that for $m\in\ell^\infty(I)$ and ${\mathcal R}\in\ell^\infty(I,B({\mathcal H}))$ hypothesis ${\mathcal C}(m,{\mathcal R})$ are valid.
Then ${\bf M}_{m{\mathcal R},{\mathcal V},{\mathcal W}}$ is invertible if and only if the ordinary Bessel multiplier ${\bf M}_{\widehat{m},\Phi,\Psi}$ is invertible, where $\Phi=\{\omega_i\varphi_{i,j}\}_{i\in I,j\in J_i}$, $\Psi=\{\upsilon_iP_{V_i}R_i\widetilde{\varphi_{i,j}}(0)\}_{i\in I,j\in J_i}$ and $\widehat{m}=\{m_{i,j}:=m_i\}_{i\in I,j\in J_i}$.
\end{theorem}
\begin{proof}
In light of the following equalities
\begin{align*}
{\bf M}_{m{\mathcal R},{\mathcal V},{\mathcal W}}x&=\sum_{i\in I}m_i\upsilon_i\omega_iP_{V_i}R_iP_{W_i}x\\&=\sum_{i\in I}m_i\upsilon_i\omega_iP_{V_i}R_i\sum_{j\in J_i}\big<P_{W_i}x,\varphi_{i,j}\big>\widetilde{\varphi_{i,j}}(0)\\
&=\sum_{i\in I}m_i\sum_{j\in J_i}\big<x,\omega_i\varphi_{i,j}\big>\upsilon_iP_{V_i}R_i\widetilde{\varphi_{i,j}}(0)\\
&={\bf M}_{\widehat{m},\Phi,\Psi},\quad\quad\quad\quad\quad\quad
\quad\quad\quad\quad\quad\quad\quad\quad\quad\quad(x\in {\mathcal H})
\end{align*}
it suffices to show that $\Phi$ and $\Psi$ are ordinary Bessel sequences. Toward this end, suppose that $x\in {\mathcal H}$. Then, observe that
\begin{align*}
\sum_{i\in I,j\in J_i}|\big<x,\upsilon_iP_{V_i}R_i\widetilde{\varphi_{i,j}}(0)\big>|^2&=
\sum_{i\in I,j\in J_i}|\big<\upsilon_iR_i^*P_{V_i}x,\widetilde{\varphi_{i,j}}(0)\big>|^2\\
&\leq\sum_{i\in I}\upsilon_i^2\sum_{j\in J_i}|\big<R_i^*P_{V_i}x,\widetilde{\varphi_{i,j}}(0)\big>|^2\\
&\leq\sum_{i\in I}\frac{\upsilon_i^2}{\alpha_i}\|{\mathcal R}\|_\infty^2\|P_{V_i}x\|^2\\
&\leq\frac{\|{\mathcal R}\|_\infty^2}{\alpha}\beta_{{\mathcal V}_\upsilon}\|x\|^2
\end{align*}
and
\begin{align*}
\sum_{i\in I,j\in J_i}|\big<x,\omega_i\varphi_{i,j}\big>|^2&=\sum_{i\in I}\omega_i^2\sum_{j\in J_i}|\big<P_{W_i}x,\varphi_{i,j}\big>|^2\\
&\leq\sum_{i\in I}\beta_i\omega_i^2\|P_{W_i}x\|^2\\
&\leq\beta\beta_{{\mathcal W}_\omega}\|x\|^2.
\end{align*}
We have now completed the proof of the theorem.
\end{proof}

%%%%%%%%%%%%%%%%%%%%%%%%%%%%%%%%%%%%%%%%%%%%%%%%%%%%%%%%%%%%%%%%%%%%%%%%%%%%%%%%%%%%%%%%

We conclude this work with the following improvement of \cite[Proposition 3.2]{ar} for $(m,{\mathcal R})$-Bessel fusion multipliers.
The details of its proof are omitted, since it can be proved by a similar argument.

\begin{proposition}
Let $({\mathcal W},\omega)$ and $({\mathcal V},\upsilon)$ be Bessel fusion sequences. Hence,
\begin{enumerate}
\item If $m\in c_0(I)$ and ${\mathcal R}\in \ell^\infty(I,B({\mathcal H}))$ is such that $\{{\rm rank}R_i\}_{i\in I}\in\ell^\infty(I)$, then $D_{m{\mathcal R}}$ is compact and, in particular, ${\bf M}_{m{\mathcal R},{\mathcal V},{\mathcal W}}$ is compact.
    \item If $m\in \ell^\infty(I)$ and ${\mathcal R}\in C_0(I,B({\mathcal H}))$ is such that $\{{\rm rank}R_i\}_{i\in I}\in\ell^\infty(I)$, then $D_{m{\mathcal R}}$ is compact and, in particular, ${\bf M}_{m{\mathcal R},{\mathcal V},{\mathcal W}}$ is compact.
        \item If $m\in \ell^p(I)$ and ${\mathcal R}\in \ell^\infty(I,B({\mathcal H}))$ is such that $\{{\rm rank}R_i\}_{i\in I}\in\ell^\infty(I)$, then $D_{m{\mathcal R}}\in C_p(\mathfrak{H},\mathfrak{H})$ and, in particular, ${\bf M}_{m{\mathcal R},{\mathcal V},{\mathcal W}}\in C_p({\mathcal H},{\mathcal H})$.

\end{enumerate}
\end{proposition}

%%%%%%%%%%%%%%%%%%%%%%%%%%%%%%%%%%%%%%%%%%%%%%%%%%%%%%%%%%%%%%%%%%%%%%%%%%%%%%%%%%%%%%%%%

\bibliographystyle{amsplain}

\begin{thebibliography}{99}

\bibitem{gav2}
{\sc F. Arabyani Neyshaburi and A. A. Arefijamaal}, Characterization and construction of $K$-fusion frames and their duals in Hilbert spaces, {\it Results Math.} {\bf 73} (2018), Art. 47, 26 pp.


\bibitem{gav3}
{\sc A. A. Arefijamaal and F. Arabyani Neyshaburi}, Some properties of dual and approximately dual of fusion frames, {\it Turkish J. Math.} {\bf 41} (2017), 1191--1203.



\bibitem{ar}
{\sc M. L. Arias and M. Pacheco}, Bessel fusion multipliers, {\it J. Math. Anal. Appl.} {\bf 348} (2008), 581--588.

\bibitem{balaz3} {\sc P. Balazs},  Basic definition and properties of Bessel
multipliers, { \it J. Math. Anal. Appl.} {\bf 325} (2007), 571--85.


\bibitem{5} {\sc P. Balazs and D. T. Stoeva}, { Represention of the inverse of a frame multiplier}, {\it J. Math. Anal. Appl.}, {\bf 422} (2015), 981--994.

\bibitem{invar}
{\sc Sh. Bishop, Ch. Heil, Y. Y. Koo and J. K. Lim},
Invariances of frame sequences under perturbations,
{\it Linear Algebra Appl.}, {\bf 432} (2010), 1501--1514.



\bibitem{6} {\sc P. G. Casazza and G. Kutyniok}, {Frames of subspaces}, {\it Contemp. Math.}, {\bf 345} (2004), 87--113.


\bibitem{ap1} {\sc P. G. Casazza, G. Kutyniok and S. Li}, {Fusion frames and distributed processing}, {\it Appl. Comput. Harmon. Anal.} {\bf 25} (2008), 114--132.



\bibitem{c} {\sc O. Christensen,} An Introduction to Frames
and Riesz Bases, Birkh\"{a}user, (2016).



\bibitem{dou} {\sc R. G. Douglas}, {On majorization, factorization and range inclusion of operators on Hilbert space}, {\it Proc. Amer. Math. Soc.}, {\bf 17} (1996), 413--415.



\bibitem{13} {\sc P. G$\check{a}$vru$\c{t}$a}, {On the duality of fusion frames}, {\it J. Math. Anal. Appl.}, {\bf 333} (2007), 871--879.


\bibitem{feichmulti1} {\sc H.G. Feichtinger and K. Nowak},
A first survey of Gabor multipliers. Advances in Gabor analysis.
Appl. Numer. Harmon. Anal., pp. 99--128. Birkh\"{a}user, Boston
(2003).


\bibitem{hein} {\sc S. B. Heineken and P. M. Morillas}, {Properties of finite dual fusion frames}, {\it Linear Algebra Appl.}, {\bf 453} (2014), 1--24.


\bibitem{14} {\sc S. B. Heineken, P. M. Morillas, A. M. Benavente and M. I. Zakowicz}, { Dual fusion frames}, {\it Arch. Math. (Basel)}, {\bf 103} (2014), 355--365.



\bibitem{ap2} {\sc S. S. Iyengar and R. R. Brooks}, {Distributed sensor networks}, Chapman, Boston Rouge, 2005.




\bibitem{j}
{\sc H. Javanshiri}, Invariances of the operator properties of frame multipliers under perturbations of frames and symbol, {\it Numer. Funct. Anal. Optim.}, {\bf 39} (2018), 571--587.

\bibitem{j}
{\sc H. Javanshiri},
Some properties of approximately dual frames in Hilbert spaces
{\it Results Math.}, {\bf 70} (2016), 475--485.



\bibitem{jc}
{\sc H. Javanshiri and M. Choubin}, Multipliers for von Neumann-Schatten Bessel sequences in separable Banach spaces, {\it Linear Algebra Appl.}, {\bf 545} (2018), 108--138.


\bibitem{jffffsss}
{\sc H. Javanshiri and A. M. Fattahi}, Continuous atomic systems for subspaces, {\it Mediterr. J. Math.}, {\bf 13} (2016), 1871--1884.



\bibitem{jffff}
{\sc A. M. Fattahi and H. Javanshiri}, Discretization of continuous frame, {\it Proc. Indian Acad. Sci. Math. Sci.}, {\bf 122} (2012), 189--202.



\bibitem{kaf}
{\sc V. Kaftal, D. R. Larson and S. Zhang}, {Operator-valued frames}, {\it Trans. Amer. Math. Soc.}, {\bf 361} (2009), 6349--6385.



\bibitem{16} {\sc G. Kutyniok, V. Paternostro and F. Philipp}, { The effect of perturbations of frame sequences and fusion frames on their duals}, {\it Oper. Matrices.}, {\bf 11} (2017), 301--336.



\bibitem{ap3} {\sc C. J. Rozell and D. H. Jahnson}, {Analysing the robustness of redundant population codes in sensory and feature extraction systems}, {\it Neurocomputing}, {\bf 69} (2006), 1215--1218.




\bibitem{msh}
{\sc M. Shamsabadi and A. A. Arefijamaal}, The invertibility of fusion frame multipliers, {\it Linear Multilinear Algebra}, {\bf 65} (2017), 1062--1072.


\bibitem{SB2} {\sc D. T. Stoeva and P. Balazs},
Invertibility of multipliers, {\it Appl. Comput. Harmon. Anal.} {\bf 33} (2012), 292--299.


\end{thebibliography}

\end{document}